\documentclass[a4paper,11pt,reqno]{amsart}
\usepackage{amsmath,amsfonts,amssymb,hyperref,amsthm,enumerate,color}
\usepackage[left=2.6cm,right=2.6cm,top=3cm,bottom=3cm,bindingoffset=0cm]{geometry}
\newtheorem{thm}{Theorem}[section]  
\newtheorem{lem}[thm]{Lemma}
\newtheorem{cor}[thm]{Corollary}
\newtheorem{prop}[thm]{Proposition}
\newtheorem{question}[thm]{Question}

\newtheorem*{case1}{Case~1}
\newtheorem*{case2}{Case~2}
\newtheorem*{claim}{Claim}
\newtheorem*{claim1}{Claim~1}
\newtheorem*{claim2}{Claim~2}
\newtheorem*{claim3}{Claim~3}
\theoremstyle{remark}
\newtheorem*{rem}{Remark}
\def\A{\mathcal{A}}
\def\B{\mathcal{B}}
\def\Z{\mathbb{Z}}
\DeclareMathOperator{\AGL}{AGL}
\DeclareMathOperator{\aut}{Aut}
\DeclareMathOperator{\atom}{Atom}
\DeclareMathOperator{\cay}{Cay}
\DeclareMathOperator{\cS}{\mathcal{S}}
\DeclareMathOperator{\dih}{Dih}
\DeclareMathOperator{\iso}{Iso}
\DeclareMathOperator{\orb}{Orb}
\DeclareMathOperator{\PGaL}{P\Gamma L}
\DeclareMathOperator{\PGL}{PGL}
\DeclareMathOperator{\PSL}{PSL}
\DeclareMathOperator{\rk}{rank}
\DeclareMathOperator{\sym}{Sym}
\DeclareMathOperator{\soc}{soc}
\DeclareMathOperator{\Span}{Span}
\DeclareMathOperator{\Sup}{Sup}
\DeclareMathOperator{\stab}{Stab}

\title{Dihedral groups of square-free order are DCI-groups}
\author[I.~Kov\'acs]{Istv\'an Kov\'acs}
\address[I.~Kov\'acs]{University of Primorska, UP IAM and UP FAMNIT, Muzejski trg 2 and Glagolja\v{s}ka 8, SI-6000 Koper, Slovenia}
\email[I.~Kov\'acs]{istvan.kovacs@upr.si}
\author[G.~Somlai]{G\'abor Somlai}
\address[G.~Somlai]{E\"otv\"os L\'or\'and University, Department of Algebra and Number Theory, 
P\'azm\'any P\'eter s\'et\'any 1/C, H-1117 Budapest, Hungary.}
\email[G.~Somlai]{gabor.somlai@ttk.elte.hu}
\thanks{The first author was supported by the Slovenian Research and Innovation Agency 
(ARIS), research program P1-0285 and research projects J1-50000 and 
N1-0391, and the second author was supported by OTKA grants 138596 and FK 142993.}
\begin{document}

\begin{abstract}
A finite group $G$ is a called a DCI-group if any two isomorphic Cayley digraphs 
of $G$ are also isomorphic via an automorphism of $G$. 
If $G$ is a non-abelian generalised dihedral DCI-group, then 
Dobson, Muzychuk, and Spiga proved that $G$ 
must be a dihedral group of square-free order (Ars Math. Contemp., 2022). In this paper, we prove that the converse statement also holds, i.\,e., all dihedral groups of square-free order are DCI-groups.
\end{abstract}

\maketitle
\section{Introduction}\label{sec:intro}

Let $G$ be a finite group and let $S$ be a subset of $G$ such that $e_G \notin S$, where 
$e_G$ denotes the identity element of $G$. The \emph{Cayley digraph} $\cay(G,S)$ is defined to have vertex set $G$ and arcs of the form $(g,sg)$, where $g \in G$ and $s \in S$. 
Any automorphism $\alpha$ of $G$ induces an isomorphism between $\cay(G,S)$ and $\cay(G,S^{\alpha})$, these digraphs are called \emph{Cayley isomorphic}. 
The subset $S$ is called a \emph{CI-subset} if for every $T \subseteq  G,\ e_G \notin T$, the relation $\cay(G,S) \cong \cay(G,T)$ implies that the two digraphs are also Cayley isomorphic. The group $G$ is called a \emph{DCI-group} if all subsets of $G \setminus \{e_G\}$ are CI-subsets, and it is called a \emph{CI-group} if the same conclusion holds for only those subsets which are closed with respect to taking inverses. 

\'Ad\'am~\cite{A} conjectured that all cyclic groups are DCI-groups. 
This was disproved soon after its publication, however, the classifications of cyclic 
CI- and DCI-groups were given only 30 years later by Muzychuk~\cite{M97}.

\begin{thm}[\cite{M97}]
A cyclic group of order $n$ is a DCI-group if and only if $n$ is square-free or twice square-free. Furthermore, the cyclic CI- but not a DCI-groups are those of order $8,\, 9$ and $18$, resp.
\end{thm}

The request for determining all finite CI-groups appeared in~\cite{BF} (see also the survey paper~\cite{L02}). Building upon the work of several mathematicians, 
the candidates of DCI-groups has been reduced to a restricted list~\cite{DMS15,LLP}. However, the verification that a particular group on this list is indeed a DCI-group is a difficult task. Further significant restriction on generalised dihedral DCI-groups
was obtained by Dobson et al.~\cite{DMS22}. For an abelian group $A$, the \emph{generalised dihedral group} $\dih(A)$ is the semidirect product 
of $A$ with the cyclic group $\langle x \rangle$ of order $2$, where $x$ acts on $A$ by conjugation as $x^{-1}ax=a^{-1}$ ($a \in A$). If $A$ is a non-trivial cyclic group of order 
$n$, then $\dih(A)$ is known as the dihedral group of order $2n$. In this paper, we denote by $D_{2n}$ the \emph{dihedral group} of order $2n$. 

\begin{thm}[\cite{DMS22}]
Let $\dih(A)$ be a generalised dihedral group over the abelian group $A$. If
$\dih(A)$ is a CI-group, then, for every odd prime $p$ the Sylow p-subgroup of $A$ has order $p$ or $9$. If $\dih(A)$ is a DCI-group, then, in addition, the Sylow 3-subgroup has order $3$.
\end{thm}

If $|A|$ is even and $\dih(A)$ is non-abelian, 
then there are involutions both inside and outside of the center 
of $\dih(A)$, implying that $\dih(A)$ is not a DCI-group.  
Therefore, the determination of the non-abelian generalised dihedral DCI-groups is reduced 
to decide whether the dihedral group $D_{2n}$ is a DCI-group for an odd square-free number $n$. This has been confirmed for special values of $n$: $n=p$~\cite{B}, $n=3p$~\cite{DMS15},  
$n \in \{pq,pqr\}$~\cite{M23}, where $p, q, r$ are pairwise distinct odd primes. 
The group $D_{2n}$ is conjectured in~\cite{DMS15,XFK} to be a DCI-group 
for every odd square-free number $n > 1$. In this paper, we show that this is indeed the case;  the following theorem is the main result of this paper.

\begin{thm}\label{thm:main}
If $n$ is an odd square-free integer larger than $1$, then $D_{2n}$ is a DCI-group. 
\end{thm}

We prove the above theorem using the \emph{Schur ring method}. 
The application of Schur rings (S-rings for short) to the study of DCI-groups dates back to~\cite{KP};   
in this paper, we apply results about S-rings obtained in~\cite{HM,SM}. For a survey 
covering various applications of S-rings in algebraic combinatorics, 
we refer the reader to~\cite{MP}.

The outline of the paper is as follows. Section~\ref{sec:S-rings} reviews the necessary   
background on S-rings. In this section, we reformulate Theorem~\ref{thm:main} as  
Theorem~\ref{thm:main-v2} in the context of S-ring theory. The proof of Theorem~\ref{thm:main-v2} is presented in 
Section~\ref{sec:proof}. This proof relies on an auxiliary result concerning a certain class of permutations groups, which is derived in Section~\ref{sec:d-groups}.
\medskip

We conclude the introduction by fixing notation used throughout the paper.
\subsection*{Notation} 
Let $G$ be a group acting on a set $X$ and $\delta$ be a partition of $X$.
\begin{itemize}
    \item The identity element of $G$ will be denoted by $e_G$ and let $G^\#=G \setminus \{e_G\}$. 
    \item For $g \in G$, let $\hat{g}$ denote the \emph{right translation} acting on $G$ as $x^{\hat{g}}=xg$ ($x \in G$). For $H \subseteq G$, let $\hat{H}=\{\hat{h} \mid h \in H\}$ and 
    \[
    \Sup(\hat{G})=\{ K \le \sym(G) \mid \hat{G} \le K \}.
    \]
    \item For $S \subseteq G$, let $S^{-1}=\{x^{-1} \mid x \in S\}$ and $\langle S \rangle$ be the subgroup generated by $S$, and $\underline{S}$ be the group ring element $\sum_{x\in S}x$ in the integer group ring $\Z G$. 
    \item For $x \in X$, denote by $G_x$ the stabilizer of $x$ in $G$, by $x^G$ the $G$-orbit of $x$, and by $\orb(G,X)$ the set of all $G$-orbits.
     \item If $Y \subseteq X$ is $G$-invariant,i.e. $Y$ is the union of $G$-orbits, then for $g \in G$, let $g^Y$ denote the permutation of $Y$ obtained by restricting $g$ to $Y$, and let $g^{Y^*}$ denote the permutation of $X$ defined by 
     \[ 
     x^{g^{Y^*}} = \begin{cases} 
     x^g & \mbox{if}~x \in Y, \\ 
     x & \mbox{if}~x \notin Y, 
     \end{cases} \quad x \in X.
     \] 
     For $H \subseteq G$, let $H^Y=\{h^Y \mid h \in H\}$ and $H^{Y^*}=\{h^{Y^*} \mid h \in H\}$. 
    \item For $g \in G$, let $\delta^g$ denote the partition of $X$ defined as $\delta^g=\{ \Delta^g \mid \Delta \in \delta\}$, and let 
    \[
    G_{\{\delta\}}=\{ g \in G \mid \delta^g=\delta \}.
    \]
    \item If $G$ is transitive on $X$ and $G=G_{\{\delta\}}$, then $\delta$ is called a \emph{system of blocks} for $G$ and its classes are referred to as \emph{blocks}. Then $\delta$ called is \emph{trivial} if its blocks are either singletons or the whole set; and \emph{minimal} if there is no non-trivial system of blocks $\delta^\prime$ for $G$ such that $\delta^\prime \ne \delta$ and $\delta^\prime$ refines $\delta$.  
    \item If $\delta$ is a system of blocks for $G$ and $g \in G$, then let $g^\delta$ denote the image of $g$ under the action of $G$ on $\delta$. Furthermore, let $G^\delta$ and $G_\delta$ denote the image and the kernel of the latter action, or more formally, 
    \[
    G^\delta=\{ g^\delta \mid g \in G\}~\text{and}~G_\delta=\{g \in G \mid 
    \forall \Delta \in \delta,~\Delta^g=\Delta\}.
    \]
\end{itemize}

\section{S-rings}\label{sec:S-rings}

Let $H$ be a finite group. A subring  $\A$ of the group ring $\Z H$ is called 
a \emph{Schur ring} (\emph{S-ring} for short) over $H$ if there exists a partition $\cS(\A)$ of $H$ such that
\begin{enumerate}[\rm (a)]
\setlength{\itemsep}{0.4\baselineskip}
    \item $\{e_H\} \in \cS(\A)$,
    \item  if $X \in \cS(\A)$ then $X^{-1} \in \cS(\A)$,
    \item $\A=\Span_{\Z}\{\underline{X} \mid X \in \cS(\A)\}$.
\end{enumerate}
This definition is due to Wielandt~\cite[Chapter~IV]{Wbook} and it is motivated by a result of Schur~\cite{Sch}, which shows that for any group 
$G \in \Sup(\hat{H})$, the $\Z$-submodule 
\[
\Span_{\Z}\{\underline{X}~\mid~X \in \orb(G_{e_H},H)\}
\]
is a subring of $\Z H$ (see also~\cite[Theorem~24.1]{Wbook}). The latter subring is called the \emph{transitivity module} over $H$ induced by $G$, denoted by $V(H,G_{e_H})$. Transitive modules are S-rings, however, the converse is not true --  
there exist S-rings which are not transitive modules~\cite{Wbook}. In this section, we 
review all concepts and results from S-ring theory needed in this paper. 
\medskip

Let $\A$ be an S-ring over $H$. 
The classes of $\cS(\A)$ are called the \emph{basic sets} of
$\A$ and the number $\rk(\A):=|\cS(\A)|$ is called the \emph{rank} of $\A$. 
A subset $X \subseteq H$ is called an 
\emph{$\A$-set} if $\underline{X} \in \A$ and a subgroup
$K \leq H$ is called an \emph{$\A$-subgroup} if $\underline{K} \in \A$. 
With each $\A$-set $X$ one can naturally associate two
$\A$-subgroups, namely, $\langle X \rangle$ and
\[
\stab(X):=\{h\in H \mid hX=Xh=X\}.
\]

Let $L \lhd U \leq H$. The section $S:=U/L$ of $H$ is called an \emph{$\A$-section} if $U$ and $L$ are $\A$-subgroups. In this case, the $\Z$-submodule of $\Z S$, defined as 
\[
\A_S:=\Span_{\Z}\{ \underline{X/L}~\mid~X \in \cS(\A), X \subseteq U\}
\]
is an S-ring over $S$. Here $X/L$ is the subset of the  
group $S=U/L$ consisting of the cosets $Lx$, where $x$ runs over $X$.  
Note that, if $\A=V(H,G_{e_H})$ and $S$ is an $\A$-section,
then $\A_S=V(S,(G^S)_{e_S})$~\cite[Proposition~2.8]{HM}. 

Suppose, in addition, that $\B$ is an S-ring over another group 
$K$ (possibly $K=H$). By a \emph{combinatorial isomorphism} (\emph{isomorphism} for short) 
from the S-ring $\A$ to the S-ring $\B$ we mean a bijective mapping $\varphi : H \to K$ such that  
\begin{enumerate}[\rm (a)]
\item $\rk(\A)=\rk(\B)=r$, 
\item there are orderings $X_1,\ldots,X_r$ and $Y_1,\ldots,Y_r$ 
of the basic sets in $\cS(\A)$ and $\cS(\B)$, resp., such that $\varphi$ is an isomorphism from $\cay(H,X_i)$ to $\cay(K,Y_i)$ for every $1 \leq i \leq r$.
\end{enumerate}
Now, $\varphi$ is called \emph{normalised} if it maps
the identity element $e_H$ to the identity element $e_K$.
If there is an isomorphism from $\A$ to $\B$, then $\A$ and $\B$ are said to be 
\emph{isomorphic}, denoted by $\A \cong \B$.
Let $\iso(\A,\B)$ denote the set of all isomorphisms from $\A$ to $\B$.
If $\varphi \in \iso(\A,\B)$ is normalised, then
$X^{\,\varphi} \in \cS(\B)$ for every basic set $X \in \cS(\A)$; moreover, 
it holds:   
\begin{equation}\label{eq:XiXj}
\forall  X, Y \in \cS(\A),~(XY)^\varphi=X^\varphi Y^\varphi.
\end{equation}

In the next proposition we collect further properties.

\begin{prop}{\rm (\cite[Proposition~2.7]{HM})}\label{prop:iso}
Let $\varphi : \A \to \B$ be a normalised isomorphism from an
S-ring $\A$ over a group $H$ to an S-ring $\B$ over a group $K$,
and let $E \leq H$ be an $\A$-subgroup.
\begin{enumerate}[{\rm (i)}]
\item The image $E^\varphi$ is a $\B$-subgroup of $K$. Moreover, the restriction
$\varphi_E : E \to E^\varphi$ is an isomorphism between $\A_E$ and
$\B_{E^\varphi}$.
\item For each $x \in H$, $(Ex)^\varphi=E^\varphi x^\varphi$.
\item If $E \lhd H$ and $E^\varphi \lhd K$,
then the mapping $\varphi^{H/E} : H/E \to K/E^\varphi$,
defined by 
\[
(Ex)^{\varphi^{H/E}}=E^\varphi x^\varphi,\quad x \in H
\] 
is a normalised isomorphism from $\A_{H/E}$ to $\B_{K/E^\varphi}$.
\end{enumerate}
\end{prop}

In this paper we are interested in isomorphisms between S-rings over the same group $H$ and set 
the following notations:
\begin{align*}
\iso(\A) &=\big\{\varphi \in \iso(\A,\A')~\mid~\A'~\text{is an S-ring over}~H\big\}, \\ 
\iso_1(\A) &=\{\varphi \in \iso(\A) \mid e_H^\varphi=e_H\}.
\end{align*}
Note that $\iso(\A,\A)$ is a subgroup of $\sym(H)$. This contains the
normal subgroup $\aut(\A)$ defined as 
\[
\aut(\A)=\bigcap_{\substack{X\in \cS(\A) \\ X \ne \{e_H\}}}\aut(\cay(H,X)),
\]
called the \emph{automorphism group} of $\A$~\cite{KP}.
It follows directly from the above definitions that for every $g \in \aut(\A)$ and 
$\alpha \in \aut(H)$, we have that $g\alpha \in \iso(\A)$. The S-ring $\A$ is called a \emph{CI-S-ring} if the following equality holds~\cite[Definition~3]{HM}:
\begin{equation}\label{eq:CI-S}
\iso_1(\A)=\aut(\A)_{e_H}\aut(H).
\end{equation}

Let $G \in \Sup(\hat{H})$ be a $2$-closed group 
(for a definition, see~\cite{DMbook}). 
Hirasaka and Muzychuk~\cite[Theorem~2.6]{HM} 
showed that the condition that $V(H,G_{e_H})$ is a CI-S-ring is equivalent to 
the condition that any two regular subgroups of $G$, which are also isomorphic to $H$, are conjugate in 
$G$. This result together with Babai's lemma~\cite[Lemma~3.1]{B} and  
the fact that the automorphism groups of Cayley digraphs are $2$-closed (see, e.\,g.,~\cite{DMMbook}) 
enable us to derive that $H$ is a DCI-group by showing that 
$V(H,G_{e_H})$ is a CI-S-ring for every group $G \in \Sup(\hat{H})$.  
Somlai and Muzychuk~\cite{SM} pointed out that the analysis of 
the S-rings $V(H,G_{e_H})$ can be reduced to those groups $G$ that are minimal with respect to the relation $\prec_{\hat{H}}$ on $\Sup(\hat{H})$ defined as follows (see~\cite[Definition~2]{HM}): 
For two groups $A, B \in \Sup(\hat{H})$, $A$ is said to be an $\mathit{\hat{H}}$\emph{-complete subgroup} of $B$, denoted by $A \prec_{\hat{H}} B$, if 
\begin{enumerate}[\rm (a)]
    \item $A \le B$,
    \item for every $\varphi \in \sym(H)$, the inclusion $(\hat{H})^\varphi \leq B$ implies that 
$(\hat{H})^{\varphi b} \leq A$ for some $b \in B$.
\end{enumerate}
The relation $\prec_{\hat{H}}$ is a partial order on 
$\Sup(\hat{H})$ and the set of minimal elements of the poset $(\Sup(\hat{H}),
\prec_{\hat{H}})$
will be denoted by $\Sup^{\min}(\hat{H})$. 
The following statement is a direct consequence of~\cite[Proposition~2.4]{SM} (see also~\cite[Corollary~3.4]{KMPRS}). 

\begin{prop}\label{prop:condi}
Let $H$ be a group and assume that $V(H,G_{e_A})$ is a CI-S-ring
for every $G \in \Sup^{\min}(\hat{H})$. Then $H$ is a DCI-group.
\end{prop}

In fact, we are going to derive Theorem~\ref{thm:main} by
showing that the condition in Proposition~\ref{prop:condi} holds whenever 
$H \cong D_{2n}$ and $n$ is odd and square-free. 

\begin{thm}\label{thm:main-v2}
Let $n$ be an odd square-free integer larger than $1$, $H \cong D_{2n}$, and 
$G \in \Sup^{\min}(\hat{H})$. Then $V(H,G_{e_H})$ is a CI-S-ring.
\end{thm}

A crucial step in the proof of the above theorem is 
to show that the group $G$ given in the theorem is 
solvable. Establishing this property will be the subject of the next section.   
The proof of Theorem~\ref{thm:main-v2} will be presented in Section~\ref{sec:proof}.

\section{$d$-groups}\label{sec:d-groups}

Song et al.~\cite{SLZ} call a permutation group a $d$-group if it 
contains a regular dihedral subgroup. We extend slightly this definition by saying 
that a permutation group of degree $n$ is a \emph{$d$-group} 
if it contains a transitive dihedral subgroup of order $n^*$, where  
\[
n^*:=\begin{cases} n & \text{if~$n$ is even}, \\
2n & \text{if~$n$ is odd}. 
\end{cases}
\]
If $n$ is even, then the $d$-group $R$ we consider is regular. On the other hand, if $n$ is odd, then the stabilizers of $R$ are of order at most 2. Further $R$ contains a cyclic group of odd order which then only intersects trivially the stabilizer of any point so its action is regular. We introduce this terminology in order to be able to prove Theorem \ref{thm:solvable} using an inductive argument. 

As in our case dihedral groups have order at least $4$, it follows immediately that $d$-groups have degree at least 
$3$. If the degree is even, then our definition of a $d$-group coincides 
with the one proposed in~\cite{SLZ}.

Let $G \le \sym(X)$ be a $d$-group of degree $n$ and $K \le G$ be any subgroup. 
We write $K \prec G$ if for every subgroup $H \le G$ such that $H \cong D_{n^*}$ and $H$ is transitive on $X$, there is an element 
$g \in G$ such that $H^g \le K$. The relation $\prec$ is a partial order on the set of 
$d$-groups contained in $\sym(X)$. Note that, for an integer $n > 1$, every group in 
$\Sup^{\min}(\hat{D}_{2n})$ is a $\prec$-minimal $d$-group.
\medskip

The purpose of this section is to prove the following theorem.

\begin{thm}\label{thm:solvable}
Every $\prec$-minimal $d$-group of square-free degree is solvable.
\end{thm}

In order to prove Theorem~\ref{thm:solvable}, we follow the the approach of Muzychuk in ~\cite{M99}, 
where he showed that every group in $\Sup^{\min}(\hat{H})$ 
is solvable if $H$ is a cyclic group (see~\cite[Theorem~1.8]{M99}). 
\medskip

Let $H$ be a group isomorphic to $D_{2n}$ and $k$ be an odd divisor $k$ of $n$. 
In what follows, we denote by $H_k$ the unique cyclic subgroup of $H$ of order $k$. 
This should not cause any confusion with the notation $D_{2k}$, which stands for the dihedral group of order 
$2k$ if $k >1$. 

\begin{lem}\label{lem:misha1}
Let $G \le \sym(X)$ be a $d$-group of degree $n > 1$ such that 
$n$ is not divisible by $4$ and let $\delta$ be a system of blocks for 
$G$ with block size $k$ such that $k < n/2$.
\begin{enumerate}[(i)]
\item The group $G^\delta$ is also a $d$-group. 
\item If $G$ is $\prec$-minimal, then $G^\delta$ is also $\prec$-minimal. 
\end{enumerate}
\end{lem}
\begin{proof}
(i): The group $G$ contains a transitive subgroup $H$ such that $H \cong D_n{^*}$. 
Note that $n^*/2$ is odd, so $H_{n^*/2}$ is the cyclic 
subgroup of $H$ of index $2$. 
For every $x \in X$, the stabilizer $H_x$ has order at most $2$, and this 
implies that $H_{n^*/2}$ is semiregular. 
Every subgroup of $(H_{n^*/2})^\delta$ is normal in $H^\delta$ and 
$H^\delta$ is transitive on $\delta$, whence $(H_{n^*/2})^\delta$ is semiregular, since the stabilizer of every point is the same normal subgroup. 
On the other hand, it is easily seen that 
$(H_{n^*/2})^\delta$ is transitive on $\delta$ if the block size $k$ is even or $n$ is odd, 
and intransitive otherwise dividing $\delta$ into two orbits. Thus, 
\begin{equation}\label{eq:order}
|(H_{n^*/2})^\delta|=\begin{cases} 
n/k & \text{if $k$ is even or $n$ is odd}, \\  
n/2k & \text{otherwise}.
\end{cases}
\end{equation}

Assume that $b \in (H \setminus H_{n^*/2}) \cap G_\delta$. 
Then for any generator $a$ of $H_{n^*/2}$ and block $\Delta \in \delta$, 
$\Delta^a=\Delta^{ab}= \Delta^{ba^{-1}}=\Delta^{a^{-1}}$, implying that $a \in G_\delta$ because 
$(H_{n^*/2})^\delta$ is semiregular and $n^*/2$ is odd. 
This, together with Eq.~\eqref{eq:order}, leads to the inequality 
$1=|(H_{n^*/2})^\delta|\ge n/2k$, contradicting the assumption that $k < n/2$.
Thus $G_\delta \cap H$ is a cyclic group, and therefore, 
$H^\delta$ is a transitive dihedral subgroup of $G^\delta$. 
Using Eq.~\eqref{eq:order}, we compute that 
$|H^\delta|=2|(H_{n^*/2})^\delta|=(n/k)^*$, and this establishes part (i). 

(ii) This can be derived using (i) following literally the proof 
of~\cite[Proposition~4.1]{M99}. 
\end{proof}

For positive integers $u$ and $v$, we write $u^i \mid\mid v$ for some integer $i \ge 1$ if 
$u^i \mid v$ and $u^{i+1} \nmid v$.
The next lemma is the analogue of~\cite[Lemma~4.3]{M99}. 

\begin{lem}\label{lem:misha2}
Let $G \le \sym(X)$ be a $\prec$-minimal $d$-group of degree $n$ and 
$\delta$ be a system of blocks for $G$ with block size $k$. 
Assume that there is an odd prime divisor $p$ of $k$ 
such that for every $\Delta \in \delta$, all orbits of a Sylow $p$-subgroup of $(G_{\delta})^{\Delta}$ are of cardinality $p^t$, where $p^t \mid\mid k$. Then 
\begin{enumerate}[(i)]
\item $G$ has a system of blocks $\varphi$ with block size $p$.
\item A Sylow $p$-subgroup of $G_{\varphi}$ is normal in $G$.
\end{enumerate}
\end{lem}
\begin{proof} 
We imitate the proof of~\cite[Lemma~4.3]{M99}.

(i): Fix a transitive subgroup $H \le G$ such that $H \cong D_{n^*}$ and let  
$\varphi=\orb(H_p,X)$. 
It is sufficient to prove that $G_{\{\varphi\}} \prec G$. 

Let $E$ be another transitive dihedral subgroup of $G$ of order $n^*$. 
Since $p^t \mid k$, it follows that both $H_{p^t}$ and $E_{p^t}$ are contained in 
$G_{\delta}$. There is an element $g \in G_{\delta}$ such that 
$H_{p^t}$ and $(E_{p^t})^g$ are contained in the same 
Sylow $p$-subgroup $P$ of $G_{\delta}$.
We claim that $E^g \le G_{\{\varphi\}}$. 
We will simply write $E$ instead of $E^g$, and hence we aim to show that 
$E \le G_{\{\varphi\}}$, provided that $E_{p^t} \le P$.  

Let $Y$ be an arbitrary $P$-orbit, in particular, write $Y=x^P$. 
Then $Y \subseteq \Delta$ for some $\Delta \in \delta$.
As $P^\Delta$ is contained in 
a Sylow $p$-subgroup of $(G_\delta)^\Delta$ and $Y$ is also a $P^\Delta$-orbit, 
we obtain the bound $|Y| \le p^t$. 
On other hand, as $x^{H_{p^t}} \subseteq Y$, $x^{E_{p^t}} \subseteq Y$, and  
$|x^{H_{p^t}}|=|x^{E_{p^t}}|=p^t$, we deduce that 
$Y=x^{H_{p^t}}=x^{E_{p^t}}$. 

Thus, both permutation groups $(H_{p^t})^Y$ and $(E_{p^t})^Y$ are regular cylic groups. 
This implies that the center $Z(P^Y)$ is contained in both of them, 
in particular, $(H_{p^t})^Y \cap (E_{p^t})^Y \ne 1$. It follows from this that 
$(H_p)^Y=(E_p)^Y$. As $Y$ was chosen to be an arbitrary $P$-orbit, we conclude that 
$\varphi=\orb(H_p,X)=\orb(E_p,X)$. This and the fact that $E_p$ is 
normal in $E$ yield that $E$ permutes the classes in $\varphi$, or equivalently, 
$E \le G_{\{\varphi\}}$, as claimed.

(ii): Write $\varphi=\{\Phi_0,\ldots,\Phi_{n/p-1}\}$. 
Then $H_p \le G_\varphi$ and we can define the subgroup $P_H$ of $G$ as 
\[
P_H:=(H_p)^{\Phi_0^*} \cdots (H_p)^{\Phi_{n/p-1}^*} \cap G_\varphi.
\]
The product $(H_p)^{\Phi_0^*} \cdots (H_p)^{\Phi_{n/p-1}^*}$ is an elementary abelian $p$-group 
of order $p^{n/p}$. Thus, $P_H$ is a $p$-group and hence it is 
contained in a Sylow $p$-subgroup $Q$ of $G_\varphi$.
For each $0 \le i \le n/p-1$, $p=|(H_p)^{\Phi_i}| \le |(P_H)^{\Phi_i}| \le |Q^{\Phi_i}| \le p$. 
It follows that $Q^{\Phi_i}=(H_p)^{\Phi_i}$ for each $0\le i\le n/p-1$, hence 
\[
Q \le Q^{\Phi_0^*} \cdots Q^{\Phi_{n/p-1}^*} \cap G_{\varphi}= (H_p)^{\Phi_0^*} \cdots (H_p)^{\Phi_{n/p-1}^*} \cap G_{\varphi} = P_H,
\] 
which directly implies that $Q=P_H$, i.\,e., $P_H$ is a Sylow $p$-subgroup of $G_{\varphi}$.
Acting by conjugation, the group $H$ permutes the subgroups $(H_p)^{\Phi_i^*}$'s, 
implying that $H \le N_G(P_H)$. 

Let us take any other transitive dihedral subgroup $F \le G$ of order $n^*$. 
The group $P_F$ can be defined as above, and as it is a Sylow $p$-subgroup of 
$G_\varphi$, there is some $g \in G_\varphi$ such that $(P_F)^g=P_H$. 
Since $g \in G_\varphi$, $(P_F)^g=P_{F^g}$, and so $F^g \le N_G(P_{F^g})=N_G(P_H)$. 
We obtain that $N_G(P_H) \prec G$, and the $\prec$-minimality of $G$ implies that 
$N_G(P_H)=G$, finishing the proof of Lemma~\ref{lem:misha2}.
\end{proof}

\begin{cor}\label{cor:misha2}
Let $G \le \sym(X)$ be a $\prec$-minimal $d$-group and 
$\delta$ be a system of blocks for $G$ with block size $k$. 
Assume that there is an odd prime divisor $p$ of $k$ 
such that for every block $\Delta \in \delta$, 
a Sylow $p$-subgroup of $(G_{\delta})^{\Delta}$ has order $p^t$, where 
$p^t \mid\mid k$. Then $G$ admits a system of blocks with block size $p$.
\end{cor}

In the proof of Theorem~\ref{thm:solvable} we shall also use the 
classification of primitive groups containing a regular cyclic group 
or a regular dihedral group. 

\begin{thm}[\cite{J,L}]\label{prim-c}
Let $G \le \sym(X)$ be a primitive group of degree $n > 1$ containing a regular 
cyclic subgroup. Then either $n=p$ is a prime and $G \le \AGL(1,p)$; or 
$G$ is $2$-transitive on $X$, and one of the following holds.
\begin{enumerate}[(a)]
\setlength{\itemsep}{0.4\baselineskip}
\item $n \ge 5$, $n$ is odd, and $G \cong A_n$. 
\item $n \ge 4$ and $G \cong S_n$. 
\item $n=\frac{q^d-1}{q-1}$, $q=r^f$ for a prime $r$, $\PGL(d,q) \le G \le \PGaL(d,q)$.
\item $n=11$ and $G \cong \PSL(2,11)$ or $G \cong M_{11}$.
\item $n=23$ and $G \cong M_{23}$.
\end{enumerate}
\end{thm}

\begin{thm}[\cite{SLZ}]\label{prim-d}
Let $G \le \sym(X)$ be a primitive group of square-free degree $n$ containing a regular 
dihedral subgroup. Then 
$G$ is $2$-transitive on $X$, and one of the following holds.
\begin{enumerate}[(a)]
\setlength{\itemsep}{0.4\baselineskip}
\item $n=22$ and $G \cong M_{22}.2$.
\item $n \ge 6$ and $G \cong S_n$.
\item $n=q+1$, $q=r^f$ for a prime $r$, $\PGL(2,q) \le G \le \PGaL(2,q)$.  
\end{enumerate}
\end{thm}

Everything is prepared to derive the main result of this section.

\begin{proof}[Proof of Theorem~\ref{thm:solvable}] 
Assume the contrary and suppose 
that $G$ is a counter example of the smallest degree $n$. 
Let $\delta=\{\Delta_0,\ldots,\Delta_{m-1}\}$ be a minimal system of blocks for $G$ with 
block size $k$ (hence $n=km$). 

Assume for the moment that $k=p$ for a prime $p$. 
Then the kernel $G_\delta$ is solvable.  
This is clear if $p=2$. If $p >2$, then a Sylow $p$-subgroup of $G_\delta$ is normal 
in $G$ due to Lemma~\ref{lem:misha2}(ii). 
This leads to the solvability of $(G_\delta)^{\Delta_i}$ 
for every $0 \le i\le m-1$. Now, as 
$G_\delta \le (G_\delta)^{\Delta_0^*} \cdots (G_\delta)^{\Delta_{m-1}^*}$, it follows that 
$G_\delta$ is indeed solvable, whence $G^\delta$ is a non-solvable group of degree $n/p$, 
in particular, $n > 2p$. However, then by Lemma~\ref{lem:misha1}(i)--(ii), 
$G^\delta$ is a $\prec$-minimal non-solvable $d$-group of degree $n/p$. This 
contradicts the assumption that $G$ is a counter example of the smallest degree.

So we may assume the $k$ is a composite number. 
Note that, if $k=2p$ for a prime $p > 2$, then it follows from Lemma~\ref{lem:misha2}(i) that 
there is a system of blocks for $G$ with block size $p$. In view of the above paragraph, 
this leads to a contradiction, and we may also assume that $k$ has at least two odd prime 
divisors. 

For $0 \le i \le m-1$, define the subgroups 
\[
K_i=(G_{\{\Delta_i\}})^{\Delta_i}~\text{and}~T_i=\soc(K_i).
\]
Note that, if $0 \le i, j \le m-1$ and $i \ne j$, then the groups $K_i$ and $K_j$ are isomorphic. 
Since $\delta$ is a minimal system of blocks, $K_i$ is a primitive group. 

For the rest of the proof we fix $H$ to be a transitive subgroup of $G$ such that $H \cong D_{n^*}$. 
Then $H_{\{\Delta_i\}}$ contains a regular cyclic or dihedral subgroup, 
and therefore, $K_i$ is one of the groups given in either 
Theorem~\ref{prim-c}~or Theorem~\ref{prim-d}. Using also that $|\Delta_i|=k$ and $k$ has distinct odd prime divisors, 
it follows that $K_i$ is $2$-transitive on $\Delta_i$ and 
$T_i$ is isomorphic to a non-abelian simple group $T$ such that either $T=A_k$, or 
$T=\PSL(d,q)$, $k=\frac{q^d-1}{q-1}$, and $q=r^f$ for a prime $r$.  
Observe that $H_\delta \ne 1$, and so $(G_\delta)^{\Delta_i}$ is a non-trivial 
normal subgroup of $K_i$. We have that $\soc((G_\delta)^{\Delta_i})=\soc(K_i)=T_i$.
\begin{claim1}
$T=A_k$. 
\end{claim1}
Suppose the contrary. Then $T=\PSL(d,q)$, $k=\frac{q^d-1}{q-1}$, and $q=r^f$ for a prime $r$. 
According to Zsigmondy's Theorem~\cite{Zs} one of the following holds:
\begin{enumerate}[(a)]
\item $r^{df}-1$ has a primitive prime divisor $s$, i.\,e., 
$s \mid (r^{df}-1)$ and $s \nmid (r^i-1)$ if $1 \le i < df$. 
\item $(df,r)=(6,2)$ or $(df,r)=(2,2^u-1)$ for an integer $u \ge 2$. 
\end{enumerate}

Suppose that (a) occurs. It is not difficult to show that $s$ is odd and it does not divide $f$. 
It follows from Theorems~\ref{prim-c} and \ref{prim-d} that 
$(G_\delta)^{\Delta_i} \le K_i \le \PGaL_d(q)$, and hence 
a Sylow $s$-subgroup of $(G_\delta)^{\Delta_i}$ has order $s^t$ such that 
$s^t \mid\mid k$. By Corollary~\ref{cor:misha2}, $G$ admits a system of blocks with block size 
$s$. Clearly, this refines $\delta$, which is in contradiction with the minimality of $\delta$. 

Now, suppose that (b) occurs. If $(df,r)=(2,2^u-1)$, then $k=2^u$, contradicting that $k$ has odd prime divisors. 
If $(df,r)=(6,2)$, then as $k$ is square-free with at least two odd prime divisors, 
we deduce that $d=3$, $r^f=4$, and $k=21$. However, then a Sylow $7$-subgroup of 
$(G_\delta)^{\Delta_i}$ has orbits of size $7$. By Lemma~\ref{lem:misha2}, $G$ admits a system of blocks with block size $7$, which refines $\delta$. This contradicts the minimality of $\delta$, and 
completes the proof of Claim~1.
\medskip

We determine next the socle of $G_\delta$. 
Let $M$ be a minimal normal subgroup of $G_\delta$ and define the set 
$[M]:=\{0 \le j \le m-1 \mid M^{\Delta_j} \ne 1\}$. 
For every $j \in [M]$, $M \cong M^{\Delta_j}=T_j$. 
Let $L$ be another minimal normal subgroup of $G_{\delta}$. 
If $j \in [M] \cap [L]$, then $[M^{\Delta_j},L^{\Delta_j}]=T_j \ne 1$. However, 
$[M,L]=1$, a contradiction. Thus $[M] \cap [L]=1$. Therefore, listing the minimal normal 
subgroups of $G_\delta$ as $M_0,\ldots,M_{l-1}$, we have that 
the subsets $[M_0],\ldots,[M_{l-1}]$ form a partition of $\{0,\ldots,m-1\}$. 
We may choose the indices so that for every $0 \le i \le l$, ]
$i \in [M_i]$. By definition, $\soc(G_{\delta})=M_0\cdots M_{l-1}$. 
Let $0 \leq i \leq m-1$. Recall that $T_i=\soc((G_\delta)^{\Delta_i})$. 
For every $g \in T_i$, define the permutation $g^*$ of $X$  acting as 
\[ 
x^{g^*} =
     \begin{cases} x^g & \mbox{if}~x \in \Delta_i, \\ x & \mbox{if}~x \notin \Delta_i, 
     \end{cases}\quad x \in X,
\] 
and let $T_i^*=\{ g^* \mid g \in T_i\}$.

\begin{claim2}
$\soc(G_\delta)=(T_0^* \cdots T_{m-1}^*) \cap G_\delta$. 
\end{claim2}
Let $N=(T_0^* \cdots T_{m-1}^*) \cap G_\delta$. 
For every $j \in 0 \le i \le l-1$ and $j \in [M_i]$, $(M_i)^{\Delta_i}=T_i$, whence 
$M_i \le \prod_{j \in [M_i]}T_i^* \le (T_0^* \cdots T_{m-1}^*)$. 
This shows that $\soc(G_\delta) \le N$. As $G_\delta$ permutes he subgroups $T_i^*$, $0 \le i \le m-1$, 
it follows that $N \lhd G_\delta$. As for every $0 \le i \le m-1$, 
$N^{\Delta_i}=T_i$, $N$ is a subdirect product of the groups $T_i$, $0 \le i \le m-1$. 
Using also that each $T_i \cong T$ and $\soc(G_\delta) \cong T^l$, we deduce that 
$N \cong T^{l'}$ for some integer $l' \ge l$. Now, if $l' > l$, then $C_N(\soc(G_\delta)) \ne 1$.
However, $C_N(\soc(G_\delta)) \le C_{G_\delta}(\soc(G_\delta))=1$ also holds, and hence 
$l=l'$, or equivalently, $\soc(G_\delta)=N$, as claimed. 
\medskip

Recall that $H \le G$ is a transitive subgroup isomorphic to $D_{n^*}$. Let $p$ be an odd  
prime divisor of $k$, and let $\varphi=\orb(H_p,X)$. The final contradiction arises from 
the following claim.

\begin{claim3}
$G_{\{\varphi\}} \prec G$.
\end{claim3}
Let $E$ be an arbitrary transitive dihedral subgroup of $G$ of order $n^*$. 
We have to find an element $g \in G$ such that $E^g \le G_{\{\varphi\}}$. 

Both groups $H_p$ and $E_p$ are contained $G_{\delta}$, and for every $0 \le i \le m-1$. 
$T_i \cong T=A_k$ due to Claim~1. 
This yields that $(H_p)^{\Delta_i}, (E_p)^{\Delta_i}  \le T_i$ and there exists an 
element $t_i \in T_i$ such that 
\begin{equation}\label{eq:ti}
((E_p)^{\Delta_i})^{t_i}=(H_p)^{\Delta_i},\quad 0 \le i \le m-1 . 
\end{equation}

Define the subsets 
\[
\Gamma_i=\bigcup_{j \in [M_i]}\Delta_j,\quad 0 \le i \le l-1.
\]

Recall that $N=(T_0^* \cdots T_{m-1}^*) \cap G_\delta$. 
By Claim~2, $N^{\Gamma_i}=\soc(G_\delta)^{\Gamma_i}=(M_0 \cdots M_{l-1})^{\Gamma_i}=M_i^{\Gamma_i}$. 
In what follows, As $M_i$ fixes pointwise $X \setminus \Gamma_i$, we regard $m_i$ as 
a permutation of $\Gamma_i$.

Fix $0 \le i \le l-1$. We have seen that the mapping $M_i \to T_i$, $m \mapsto m^{\Delta_i}$ 
($m \in M_i$) is an isomorphism. Let $m_i \in M_i$ be the element satisfying $m_i^{\Delta_i}=t_i$. We claim that 
\begin{equation}\label{eq:mi}
((E_p)^{\Gamma_i})^{m_i}=(H_p)^{\Gamma_i}.
\end{equation}
Indeed, as $\Delta_i \subseteq \Gamma_i$, Eq.~\eqref{eq:mi} is equivalent to the equality 
\[ 
((E_p)^{\Delta_i})^{m_i^{\Delta_i}}=(H_p)^{\Delta_i},
\]
which is true by Eq.~\eqref{eq:ti} and the fact that $m_i^{\Delta_i}=t_i$. 

Now, let $g=m_0 \ldots m_{l-1}$. Then $g \in N$, and for every $0 \le i \le l-1$, we can write 
\[
((E_p)^g)^{\Gamma_i}=((E_p)^{\Gamma_i})^{m_i}=(H_p)^{\Gamma_i},
\]
which shows that $(E_p)^g=H_p$. Consequently, $\orb((E_p)^g,X)=\orb(H_p,X)=\varphi$, implying 
that $E^g \le G_{\{\varphi\}}$, as required.
\end{proof}

We conclude the section with a useful property of $\prec$-minimal $d$-groups.

\begin{prop}\label{prop:Hn-is-block}
Let $G \le \sym(X)$ be a $\prec$-minimal $d$-group of square-free degree $2n$,  
$H$ be a regular subgroup of $G$ such that $H \cong D_{2n}$. 
Then $\orb(H_n,X)$ is a system of blocks for $G$. 
\end{prop}
\begin{proof}
Let $\varphi=\orb(H_n,X)$. It is sufficient to show that $G_{\{\varphi\}} \prec G$. 
By Theorem~\ref{thm:solvable}, $G$ is a solvable, hence due to Hall's theorem, $H_n$ is 
contained in a Hall $2'$-subgroup $L$ of $G$. Since $|L|$ is odd, $L$ is intransitive 
and $\varphi=\orb(L,X)$. Let $E$ be another regular subgroup of $G$ such that $E \cong D_{2n}$. 
There exists $g \in G$ such that $(E_n)^g \le L$. This implies in turn that,  
$\orb((E_n)^g,X)=\orb(L,X)=\varphi$, $E^g \le G_{\{\varphi\}}$, and so 
$G_{\{\varphi\}} \prec G$, as required. 
\end{proof}
\section{Proof of Theorem~\ref{thm:main-v2}}\label{sec:proof}

Theorem~\ref{thm:main-v2} will be derived through a series of lemmas. 
In the following lemmas we assume that 
\begin{itemize}
    \item $H$ is a group isomorphic to $D_{2n}$, where $n$ is an odd and square-free,  
    \item $G \in \Sup(\hat{H})$ is a solvable group, and 
    \item $\A=V(H,G_{e_H})$.
\end{itemize}
Furthermore, let $\pi(n)$ denote the set of prime divisors of $n$. 

\begin{lem}\label{lem:sylow}
Let $\delta$ be a system of blocks for $G$ with block size $p$ for a prime 
$p \in \pi(n)$, and $P$ be a Sylow $p$-subgroup of $G_\delta$.  
Then $\hat{H}_p \le P$ and $P=O_p(G)$.
\end{lem}
\begin{proof}
Note that $\delta=\orb(\hat{H}_p,H)$, in particular, $\hat{H}_p \le G_\delta$. 
Write $\delta=\{\Delta_0,\ldots,\Delta_{m-1}\}$ (hence $2n=mp$). 
For every $0 \le i \le m-1$, since $G$ is solvable, it follows that
$(G_{\{\Delta_i\}})^{\Delta_i} \le N_{\sym(\Delta_i)}((\hat{H}_p)^{\Delta_i})$. 
Clearly, 
\[
G_\delta \le \mathcal{G}:=(G_\delta)^{\Delta_0^*}  \cdots (G_\delta)^{\Delta_{m-1}^*}.
\]
It is easily seen that $\mathcal{P}:=(\hat{H}_p)^{\Delta_i^*} \cdots 
(\hat{H}_p)^{\Delta_{m-1}^*}$ is the Sylow 
$p$-subgroup of $\mathcal{G}$, which is normal in $\mathcal{G}$. 
These imply that $P=G_\delta \cap \mathcal{P}$ 
and $P$ is characteristic in $G_\delta$. 
Using also that $G_\delta \triangleleft G$, we obtain that $P \triangleleft G$, 
and so $P \le O_p(G)$. Also, $\hat{H}_p \le G_\delta \cap \mathcal{P}=P$. 
 
On the other hand, $\orb(O_p(G),H)=\delta$, by which, $O_p(G) \le G_\delta$. As $P$ is the unique Sylow $p$-subgroup of $G_\delta$, $O_p(G) \le P$ also holds, and we conclude that $P=O_p(G)$. 
\end{proof}

\begin{lem}\label{lem:H(p)}
Let $\delta$ be a system of blocks for $G$ with block size $p$ for a prime 
$p \in \pi(n)$, and let 
\[
H(p)=\{ x \in H \mid O_p(G)_x=O_p(G)_{e_H}\}.
\]
Then $H(p)$ is an $\A$-subgroup and 
\[
\forall X \in \cS(A),~X \not\subset H(p) \implies H_p \le \stab(X).
\]
\end{lem}
\begin{proof}
Write $N=O_p(G)$. 
By Lemma~\ref{lem:sylow}, $\hat{H}_p \le N$. 
Define the relation $\sim$ on $H$ as for $x ,y \in H$, let $x \sim y$ if and only 
if $N_x = N_y$. It is not difficult to show, using that $N \triangleleft G$, that $\sim$ is a 
$G$-invariant equivalence relation. 
This implies that the partition of $H$ induced by $\sim$ is a system of blocks (see \cite[Exercise~1.5.4]{DMbook}). As $H(p)$ is the block in the latter system 
containing the identity element $e_H$, $H(p)$ is an $\A$-subgroup. 

Now, let $X \in \cS(\A)$ such that $X \not\subset H(p)$. Then for $x \in X$, 
$N_{e_H} \ne N_x$. 
This implies that $N_{e_H}$ is transitive on the coset $H_px$, hence $H_px \subseteq X$. As this holds for any $x \in X$, we conclude that $H_p \le \stab(X)$. 
\end{proof}

We call a basic set $X$ of $\A$ an {\em atom} if $X=Lx$ for some subgroup 
$L \le H_n$ and element $x \in H \setminus H_n$. 
The set of all atoms of $\A$ will be denoted by $\atom(\A)$. 

\begin{lem}\label{lem:atom}
If $H_n$ is an $\A$-subgroup, then $\atom(\A) \ne \emptyset$. 
\end{lem}
\begin{proof}
Let $G^+$ be the setwise stabilizer of $H_n$ in $G$. 
Since $H_n$ is an $\A$-subgroup, $\{H_n,H \setminus H_n\}$ is a system of blocks for 
$G$, hence $G^+$ has index $2$ in $G$. 
Let $K$ be the kernel of the action of $G^+$ on $H_n$. 
Then $\orb(K,H \setminus H_n)$ form a non-trivial 
$G^+$-invariant partition. As $\hat{H}_n \le G^+$, this forces the existence of 
a subgroup $H_m \le H_n$ such that $\orb(K,H \setminus H_n)=
\orb(\hat{H}_m,H \setminus H_n)$. 
This implies the following property of basic sets outside $H_n$, which will be used a couple of times in the rest of the proof.  
\begin{equation}\label{eq:Cm}
\forall X \in \cS(\A),~X \not\subset H_n \implies H_m \le \stab(X).   
\end{equation} 

We prove the lemma using induction on $|\pi(n)|$. 

Assume first that $n$ is a prime. If $G^+$ is unfaithful on $H_n$, i.\,e., $K \ne 1$, then 
it follows from the property in~\eqref{eq:Cm} that $H \setminus H_n$ is an atom. 
Assume that $G^+$ is faithful on $H_n$. Then $G^+ \le \AGL(1,p)$, and one can quickly verify 
that every two faithful actions of $G^+$ of degree $p$ are  
equivalent. This implies that $G_{e_H}=(G^+)_{e_H}=(G^+)_x$ for some 
$x \in H \setminus H_n$ (see~\cite[Lemma~1.6B]{DMbook}), which further yields that 
$\{x\}$ is an atom of $\A$.

Now, assume that $n$ is a composite number and the lemma is true when $n$ is replaced with any odd square-free number larger than $1$ and having less number of prime divisors than $n$. 
If $G^+$ is unfaithful on $H_n$, then $H_m \ne 1$ in~\eqref{eq:Cm}. 
If $m=n$, then $H \setminus H_n$ is an atom. 
If $m < n$, then consider the S-ring $\A_{H/H_m}$. It is a transitivity module induced by a solvable group and $H_n/H_m$ is an $\A_{H/H_m}$-subgroup, whence it possesses an atom due to the induction hypothesis. It follows from \eqref{eq:Cm} that this atom lifts to 
an atom of $\A$. For the rest of the proof we assume that $G^+$ is faithful on $H_n$. 

As $G$ is solvable, there is a prime divisor $p$ of $|G|$ such that $O_p(G) \ne 1$. 
The $O_p(G)$-orbits form a system of blocks for $G$ whose blocks are the same as 
the $\hat{B}$-orbits for 
some subgroup $B < H$ such that $|B|=p$. If $p=2$, then $B=\langle x\rangle$ for 
some element $x \in H \setminus H_n$. 
It follows that $\{x\}$ is an atom of $\A$. 
We may assume that $O_2(G)=1$. Then $F(G)=\prod_{p \in \pi(n)}O_p(G)$, where $F(G)$ is the Fitting subgroup of $G$. We distinguish two cases according to whether 
$|F(G)|$ is square-free or not.  

\begin{case1}
$|F(G)|$ is not square-free. 
\end{case1}
There exists a prime $p \in \pi(n)$ such that $p^2$ divides $|F(G)|$. 
Consider the group $H(p)$ defined in Lemma~\ref{lem:H(p)}. 
The condition that $|O_p(G)| > p$ forces that $H(p) < G$, since otherwise the action of $O_p(G)$ is semiregular which would imply $O_p(G)=p$. 
Therefore, if $H(p)$ is a dihedral subgroup, then the induction hypothesis yields that $\A_{H(p)}$ has an atom and we are done. If $H(p) \le H_n$, then it follows from Lemma~\ref{lem:H(p)} that every atom of $\A_{H/H_p}$ 
lifts to an atom of $\A$, and we are done in this case as well.

\begin{case2}
$|F(G)|$ is square-free. 
\end{case2}
Choose a prime divisor $p$ of $|F(G)|$. Then $|O_p(G)|=p$, and it follows from Lemma~\ref{lem:sylow} that  
$\hat{H}_p=O_p(G) \le \hat{H}_n$. Thus $F(G) \le \hat{H}_n$. On the other hand, as 
$C_G(F(G)) \le F(G)$ (see \cite[Theorem~6.1.3]{Gbook}), we can 
write $\hat{H}_n \le C_G(F(G)) \le F(G) \le \hat{H}_n$. It follows that 
$F(G)=\hat{H}_n$, in particular, $\hat{H}_n \triangleleft G$. In particular we obtain that $F(G)$ is the product of subgroups of prime order and all of these subgroups are $\mathcal{A}$-subgroups. This immediately implies that every subgroup of $H_n$ is an $\mathcal{A}$-subgroup.

Let $p$ be the largest prime divisor of $n$ and let $n=pm$. 
Then $H_m$ is an $\A$-subgroup, and $\A_{H/H_m}$ is an S-ring over 
$H/H_m$, which is a dihedral group of order $2p$. 
If $H/H_m \setminus H_n/H_m$ is a basic set of $\A_{H/H_m}$, then 
every basic set of $\A$ outside $H_n$ intersects any two $H_m$-cosets outside 
$H_n$ at the mumber of elements 
see~\cite[p.~21]{EP14}.
in particular, it has size divisible by $p$. 
On the other hand, since $G^+$ is faithful on $H_n$ and 
$\hat{H}_n \triangleleft G$, it follows that  $|G_{e_H}|$ divides 
$\phi(n)$, where $\phi$ denotes the Euler's totient function, and hence 
$\gcd(p,|G_{e_H}|)=1$. This contradicts the fact that $\A$ has basic sets 
of size divisible by $p$. We obtain that $\A_{H/H_m}$ has a basic set outside 
$H/H_m \setminus H_n/H_m$ of 
size $1$. Equivalently, $H$ contains a dihedral $\A$-subgroup $B$ of order $2m$. 
Applying the induction hypothesis to $\A_{B}$, we see that $\A_{B}$ possesses an atom, and hence so does $\A$. 
\end{proof}

\begin{rem}
It is worth noting that the above lemma does not hold if $G$ is not a solvable group. 
A counter example arises from the incidence graph of the Desarguesian projective plane 
$\mathrm{PG}(2,q)$, where $q$ is a prime power. Denote this graph by $\Gamma(q)$ and let 
$A=\aut(\Gamma(q))$ (i.\,e., $\Gamma(q)$ is the bipartite graph, whose biparts represent the points and the lines, resp., and the edges are defined according to the incidence relation of the plane). It is known that $A$ contains a regular dihedral group $D$ of order $2(q^2+q+1)$. 
Identifying the vertex set of $\Gamma(q)$ with $D$ and letting 
$\B=V(D,A_{e_D})$, we obtain that $C$ is a $\B$-subgroup, where 
$C$ denotes the cyclic subgroup of $D$ of order $q^2+q+1$, and 
$D \setminus C$ splits into two basic sets of sizes 
$q+1$ and $q^2$, resp., implying that $\B$ has no atoms. 
\end{rem}

We say that a dihedral $\A$-subgroup $B \le H$ is a \emph{minimal dihedral $\A$-subgroup} if it contains no dihedral $\A$-subgroup of $H$ apart from itself. If $H_n$ is an $\A$-subgroup, then it follows from Lemma~\ref{lem:atom} that there is a
one-to-one correspondence between the atoms of $\A$ and the minimal dihedral 
$\A$-subgroups of $H$. Namely, if $X$ is an atom, then $\langle X \rangle$ is a minimal dihedral $\A$-subgroup; and conversely, if $B$ is a minimal dihedral $\A$-subgroup, then 
$B \setminus H_n$ is an atom. 
\medskip

From now on, and throughout the remainder of the section, we assume that $G \in \Sup^{\min}(\hat{H})$, or 
equivalently, $G$ is a $\prec$-minimal $d$-group.
Then, by Proposition~\ref{prop:Hn-is-block}, $H_n$ is an $\A$-subgroup, 
hence $\atom(\A) \ne \emptyset$ due to due to Lemma~\ref{lem:atom}. 
In view of Theorem~\ref{thm:solvable}, Theorem~\ref{thm:main-v2} follows if we show that 
$\A$ is a CI-S-ring. 
\medskip

In the following lemma, we refine the statements in Lemmas~\ref{lem:sylow} and \ref{lem:H(p)}. 

\begin{lem}\label{lem:mini}
Let $\delta$ be a system of blocks for $G$ with block size $p$ for a prime 
$p \in \pi(n)$. If $G$ is $\prec$-minimal, then the following conditions hold.
\begin{enumerate}[(i)]
\item $O_p(G)$ is the Sylow $p$-subgroup of $G$.
\item  Let $H(p)$ be the $\A$-subgroup defined in Lemma~\ref{lem:H(p)}. Then 
 \[
\forall X \in \cS(A),~X \not\subset H(p) \iff H_p \le \stab(X).
\]
\end{enumerate}
\end{lem}
\begin{proof}
(i): Let $P$ be a Sylow $p$-subgroup of $G$. 
In view of Lemma~\ref{lem:sylow}, it is sufficient to show that $P < G_\delta$. 
This is clearly the case if $n=p$, hence assume that $n > p$.
It follows from $\prec$-minimality and Hall's theorem (see \cite[Theorem~6.1.4]{Gbook}) 
that the odd prime divisors of $|G|$ are exactly those in 
$\pi(n)$. Consider the image $G^\delta$ of $G$ induced by its action on $\delta$. 
By Lemma~\ref{lem:misha2}(ii), $G^\delta$ is also $\prec$-minimal, and thus 
$p$ does not divide $|G^\delta|$, hence $P \le G_\delta$, as required.

 (ii): Fix an arbitrary basic set $X \in \cS(\A)$. In view of Lemma~\ref{lem:H(p)}, it is sufficient to show that  
 $H_p \not\le \stab(X)$ provided that $X \subset H(p)$.  
Assume on the contrary that $X \subset H(p)$ and $H_p \le \stab(X)$.  
As $X$ is a $G_{e_H}$-orbit, the orbit-stabiliser lemma yields that 
$|G_{e_H}|=|G_{e_H,x}|\cdot |X|$, where $x \in X$. 
As $H_p \le \stab(X)$, $p$ divides $|X|$ and hence the Sylow $p$-subgroup of $G_{e_H,x}$ is properly contained in 
the Sylow $p$-subgroup of $G_{e_H}$, by which, $P_{e_H,x}=P_{e_H} \cap P_{x} < P_{e_H}$. 
However, as $x \in H(p)$, $P_{e_H}=O_p(G)_{e_H}=O_p(G)_x=P_x$, a contradiction. 
\end{proof}

Next, we discuss some elementary properties of the atoms of $\A$. 
Let $X \in \atom(\A)$. 
It follows directly from the definition that $X=\stab(X)x$ for some $x \in H \setminus H_n$.
Thus, for any $c ,d \in H_n$,
\begin{equation}\label{eq:stabX}
Xc=Xd \iff cd^{-1} \in \stab(X).    
\end{equation}

Suppose that $H_p$ is an $\A$-subgroup for some $p \in \pi(n)$ and $n \ne p$. 
Then 
\begin{equation}\label{eq:quo}
X/H_p \in \atom(\A_{H/H_p})~\text{and}~\stab(X/H_p)=H_p\stab(X)/H_p.
\end{equation}

Let $\varphi \in \iso_1(\A)$ be a normalised isomorphism. By definition, 
$X=Lx$, where $L = \ stab(X) \le H_n$ and $x \in H \setminus H_n$. 
It follows from Proposition~\ref{prop:iso}(i)--(ii) that $L^\varphi=L$ since the image of an $\mathcal{A}$-subgroup is an $\mathcal{A}$-subgroup and there is just one subgroup of $H$ of order $|L|$. Thus we have
$X^\varphi=L^\varphi x^\varphi=L x^\varphi$.
From this, define $c=x^{-1}x^\varphi$, so $X^\varphi=Lxc=Xc$. Note that, since $\varphi$ maps 
$H \setminus H_n$ to 
itself, it follows that $c \in H_n$.
In the next lemma we show that the element $c$ depends only on the isomorphism 
$\varphi$, but not on 
the particular atom $X$, provided that $\varphi$ fixes setwise each basic set contained in $H_n$. For this purpose, we introduce some further notation. 

For any $c \in H_n$, there is an automorphism of $H$ that fixes pointwise $H_n$ and acts on 
$H \setminus H_n$ as $\hat{c}$. We denote this automorphism by $\sigma_c$.  
For $p \in \pi(n)$ and $x \in H_n$, let $x_p$ and $x_{p'}$ denote the projections of $x$ onto the $p$-component $H_p$ and
its complement $H_{n/p}$, resp., i.\,e., $x_p \in H_p$, $x_{p'} \in H_{n/p}$, and 
$x=x_px_{p'}$.  

\begin{lem}\label{lem:f-c}
Suppose that $G$ is $\prec$-minimal and let $\varphi \in \iso_1(\A)$ be an isomorphism, which fixes setwise each basic set of $\A$ contained in $H_n$. 
Then there exists an element $c \in H_n$ such that 
\begin{equation}\label{eq:Xc}
\forall X \in \atom(\A),~X^\varphi=Xc.
\end{equation}
\end{lem}
\begin{proof}
We proceed by induction on $|\pi(n)|$. 
Assume first that $n$ is a prime. 
Eq.~\eqref{eq:Xc} holds trivially if $H \setminus H_n$ is a basic set. 
Otherwise, $\A$ has an atom in the form $\{x\},\,  x \in H \setminus H_n$. 
Let $\psi=\varphi\sigma_{c}^{-1}$, where $c$ is defined as $c=x^{-1}x^\varphi$. 
It follows that $\psi \in \iso_1(\A)$ and $x^\psi=x^{\varphi\sigma_c^{-1}}=x^\varphi c^{-1}=x$.  
If $Y$ is an arbitrary basic set contained in $H_n$, then $Y^\psi=Y^\varphi=Y$. 
Let $X$ be an arbitrary atom of $\A$. Then for $Y:=xX$, $Y \in \cS(\A)$ such that $Y \subset H_n$. Therefore, $X^\psi=(xY)^\psi=x^\psi Y^\psi=xY=X$. This yields $X^\varphi=X^{\psi\sigma_c}=Xc$, i.\,e., Eq.~\eqref{eq:Xc} holds. 

Now, assume that $n$ is a composite number and the lemma is true when  
$n$ is replaced with any odd square-free number larger than $1$ 
and having less number of prime divisors than $n$. 
As discussed above, for every $X \in \atom(\A)$, there exists an element $c_X \in H_n$ such that $X^\varphi=Xc_X$. We aim to show that $c_X$ can be chosen to be a constant.

If there is an atom of $\A$ of size $1$, then the argument used in the prime case above can be 
copied. In the rest of the proof we assume that each atom of $\A$ has size larger than $1$. 
In view of this and Lemma~\ref{lem:atom}, there is an atom $Lb$ such that $1 < L \le H_n$ and $b \in H \setminus H_n$. 
Let $H_p$ be a minimal $\A$-subgroup contained in $\stab(Lb)=L$, and $H(p)$ be the 
$\A$-subgroup defined in Lemma~\ref{lem:H(p)}. 
Let $\varphi^*=\varphi^{H/H_p}$, the isomorphism of 
the S-ring $\A_{H/H_p}$ defined in Proposition~\ref{prop:iso}(iii). 
Clearly, $\varphi^*$ fixes setwise the basic sets of $\A_{H/H_p}$ contained in $H_n/H_p$, and 
if $X$ is an atom of $\A$, then $X/H_p$ is an atom $\A_{H/H_p}$, see~\eqref{eq:quo}. Then, we have that 
\[
(X/H_p)^{\varphi^*}=(X^\varphi)/H_p=(Xc_X)/H_p=(X/H_p)\;  H_p(c_X)_{p'},
\]
where $H_p(c_X)_{p'}$ is regarded as an element in the quotient group $H_n/H_p$. 
The S-ring $\A_{H/H_p}$ is equal to the transitivity module 
$V(H/H_p,(G^\delta)_{e_{H/H_p}})$, where $\delta$ is the system of blocks consisting of the 
$\hat{H}_p$-orbits. In view of the fact that $G^\delta$ is $\prec$-minimal,  the induction hypothesis can be applied to $\A_{H/H_p}$.  
As a result, we can choose an element $d \in H_{n/p}$ such that   
\[
(X/H_p)\; H_p(c_X)_{p'}=(X/H_p)\; H_p d.
\]
Using the equivalence in~\eqref{eq:stabX}, we get the condition 
$H_p (c_X)_{p'}d^{-1} \in \stab(X/H_p)$. As $\stab(X/H_p)=H_p\stab(X)/H_p$, see \eqref{eq:quo}, this implies that 
the coset $H_p(c_X)_{p'}d^{-1}$ is contained in the subgroup $H_p\stab(X)$, whence 
\[
c_X\, ((c_X)_pd)^{-1}=(c_X)_{p'}d^{-1} \in \stab(X) \cap H_{n/p},
\]
and therefore, $Xc_X=X(c_X)_pd$, see~\eqref{eq:stabX}. 
Now, as $d$ does not depend on the particular atom $X$, the following identity holds.  
\begin{equation}\label{eq:XcXpd}
\forall X \in \atom(\A),~X^\varphi=X(c_X)_pd.
\end{equation}

Now, if $X \in \atom(\A)$ such that $X \not\subset H(p)$, then 
$H_p \le \stab(X)$ due to Lemma~\ref{lem:mini}(ii), hence we obtain that 
$X^\varphi=Xd$. 
In the case where $H(p) \le H_n$, this shows that 
Eq.~\eqref{eq:Xc} holds for $c=d$. 

Thus we may assume that $H(p)$ is a dihedral subgroup. 
Let $X' \in \atom(\A)$ such that $X' \subset H(p)$. 
We have seen that $\stab(X') \ne 1$, and there is a minimal 
$\A$-subgroup $H_q$ contained in $\stab(X')$. By Lemma~\ref{lem:mini}(ii), $q  \ne p$. 
Now, one can copy the previous argument after switching $c_X$ with $(c_X)_pd$, where $X$ runs over the set 
$\atom(\A)$, and also replacing $p$ with $q$. This leads to the existence of an element 
$f \in H_{n/q}$ such that the identity in~\eqref{eq:XcXpd} becomes the following one:
\[ 
\forall X \in \atom(\A),~X^\varphi=X((c_X)_pd)_q f=Xd_qf.
\] 
We conclude that Eq.~\eqref{eq:Xc} holds for $c=d_qf$.  
\end{proof}

We are one step away from showing  
that $\A$ is a CI-S-ring, and hence deriving Theorem~\ref{thm:main-v2}. 
Let $\varphi \in \iso_1(\A)$ be an 
arbitrary isomorphism. By Eq.~\eqref{eq:CI-S}, $\A$ is a CI-S-ring if and only if 
there exists an automorphism $\alpha \in \aut(H)$ such that 

\begin{equation}\label{eq:CI}
\varphi\alpha \in \aut(\A).
\end{equation}

As $\A_{H_n}$ is a CI-S-ring (see~\cite{KMPRS}), there exists an automorphism $\beta \in \aut(H)$ such that $X^\varphi=X^\beta$ for each basic set $X \in \cS(\A)$,  
$X \subset H_n$. Then, due to Lemma~\ref{lem:f-c}, there exists $c \in H_n$ such that $X^{\varphi\beta^{-1}}=Xc$ for each atom $X \in \atom(\A)$. Consequently, $\varphi\beta^{-1}\sigma_c^{-1}$ 
is an isomorphism in $\iso_1(\A)$, which fixes setwise each basic set of $\A$ contained in $H_n$ and each atom of $\A$.  
We show that the statement in~\eqref{eq:CI} holds for $\alpha=\beta^{-1}\sigma_c^{-1}$ 
by proving the following lemma.

\begin{lem}\label{lem:f-fix}
Suppose that $G$ is $\prec$-minimal and $\varphi \in \iso_1(\A)$ is an isomorphism, which fixes setwise each basic set of $\A$ contained in $H_n$ and each atom of $\A$. Then 
$\varphi \in \aut(\A)$, or equivalently,  
\begin{equation}\label{eq:goal}
\forall X \in \cS(\A),~X^\varphi=X.
\end{equation}
\end{lem}
\begin{proof}
We proceed by induction on $|\pi(n)|$. 
If $n$ is a prime, then $H \setminus H_n$ is an atom, or there is an atom of size $1$. In either case, it is easy to show that Eq.~\eqref{eq:goal} holds. 

Assume that $n$ is a composite number and the lemma is true when $n$ is replaced by any odd square-free number larger than $1$ and having less number of prime divisors than $n$. It is easy to see that Eq.~\eqref{eq:goal} holds if at least one of the atoms of $\A$ has size $1$. 

In the rest of the proof we assume that each atom has size larger than $1$. Consequently, each minimal $\A$-subgroup has odd prime order. Denote by $\pi^*(n)$ the set of primes $p \in \pi(n)$ for which $H_p$ is an $\A$-subgroup. Fix $p \in \pi^*(n)$ and let $H(p)$ be the $\A$-subgroup defined in Lemma~\ref{lem:H(p)}. 
We consider the S-ring $\A_{H/H_p}$ and let $\varphi^*=\varphi^{H/H_P}$, 
the isomorphism of $\A_{H/H_p}$ defined in Proposition~\ref{prop:iso}(iii). 

\begin{claim}
$\varphi^*$ fixes setwise 
\begin{enumerate}[(a)]
\item each basic set contained in $H_n/H_p$, and 
\item each atom of $\A_{H/H_p}$. 
\end{enumerate}
\end{claim}
Part (a) is obvious, and to verify part (b), recall that the atoms of $\A_{H/H_p}$ correspond to the minimal dihedral $\A_{H/H_p}$-subgroups of $H/H_p$. To settle part (b), it is sufficient to show that  
$\varphi^*$ fixes setwise each minimal dihedral $\A_{H/H_p}$-subgroups. Such a group 
can be written in the form $B/H_p$, where $B$ is a dihedral $\A$-subgroup and 
$H_p \le B$. Let $K$ be a minimal dihedral $\A$-subgroup 
contained in $B$. Then $K^\varphi=K$ because of the assumption that 
$\varphi$ fixes setwise 
the atom $K \setminus H_n$, and $KH_p=B$ because of the minimality of 
$B/H_p$. These yield that $(B/H_p)^{\varphi^*}=B^\varphi/H_p=(KH_p)^\varphi/H_p=B/H_p$, and 
this completes the proof of the claim. 
\medskip

Now, the induction hypothesis can be applied to the S-ring 
$\A_{H/H_p}$ and $\varphi^*$, and this yields that 
$X^\varphi/H_p=(X/H_p)^{\varphi^*}=X/H_p$.  Therefore, if 
$X \not\subset H(p)$, then $H_p \le \stab(X)$ due to 
Lemma~\ref{lem:mini}(ii), hence Eq.~\eqref{eq:goal} holds for $X$. 
More generally, Eq.~\eqref{eq:goal} holds whenever 
\[
X \in \bigcup_{p \in \pi^*(n)}\overline{H(p)}=\overline{\bigcap_{p \in \pi^*(n)}H(p)}.
\]
For a subset $S \subseteq H$, we denote by $\overline{S}$ the complement of $S$ 
in $H$. Let $H_0=\bigcap_{p \in \pi^*}H(p)$. Now, we are done if show that $H_0 \le H_n$.  

Assume on the contrary that $H_0 \not\le H_n$, and choose an atom $X' \in \atom(\A)$ contained 
in $H_0$. Since $\stab(X') \ne 1$,  
there is a minimal $\A$-subgroup $C_q$ contained in $\stab(X')$. 
By Lemma~\ref{lem:mini}(ii), $X'$ is outside $H(q)$. However, this 
contradicts that $q \in \pi^*(n)$ and $X' \subset H_0$. 
\end{proof}
We have modified a proof of Muzychuk in order to prove Theorem \ref{thm:solvable} which guarantees that the $\prec$-minimal overgroups of a regular dihedral groups of square-free order  are solvable and as it was proved originally by Muzychuk the same holds for regular cyclic groups of square-free order. Notice that this was one of the most important steps in Muzychuk's and our proof and the solvability is used by Morris \cite{M23} as well.

We raise the question which groups have a similar property:
\begin{question}
    What are the groups $G$ such that the $\prec$-minimal overgroups of the regular copy $\hat{G}$ are all solvable. 
\end{question}



\begin{thebibliography}{mmm}
\bibitem{A}
 A. \'Ad\'am, 
Research problem 2-10, 
J. Combin. Theory 2 (1967), 393.
%
\bibitem{B}
L. Babai,
Isomorphism problem for a class of point symmetric structures,
Acta Math. Acad. Sci. Hun. 29 (1977), 329--336.
%
\bibitem{BF}
L. Babai, P. Frankl, 
Isomorphisms of Cayley graphs I, 
in: Combinatorics (Proc. Fifth
Hungar. Colloq., Keszthely, 1976), vol.~I, Colloq. Math. Soc. János
Bolyai 18, North-Holland, Amsterdam--New York, 1978, 35--52.
%
\bibitem{DMbook}
J.\,D. Dixon and B. Mortimer,
Permutation groups,
Graduate Text in Mathematics 163, Springer-Verlag, New York 1996.
%
\bibitem{D}
E. Dobson, 
On the Cayley isomorphism problem, 
Discrete Math. 247 (2002), 107--116.
%
\bibitem{DMMbook}
T. Dobson, A. Malni\v{c}, and D. Maru\v{s}i\v{c}, 
Symmetry in Graphs, 
Cambridge University Press, 2022
%
\bibitem{DMS15}
E. Dobson, J. Morris, P. Spiga, 
Further restrictions on the structure of finite DCI-groups: an addendum, 
J. Algebraic Combin. 42 (2015), 959--969.
%
\bibitem{DMS22}
T. Dobson, M. Muzychuk, P. Spiga,
Generalised dihedral CI-groups, 
Ars Math. Contemp. 22 (2022), \#P2.07. 
%
\bibitem{EP14}
S. Evdokimov, I. Ponomarenko,
Schur rings over a product of Galois rings,
Beitr. Algebra Geom. 55 (2014), 105--138.
%
\bibitem{Gbook}
D. Gorenstein,
Finite groups - 2nd edition, 
Chelsea Publishing Company, New York 1980.
%
\bibitem{HM}
M. Hirasaka, M. Muzychuk,
An elementary abelian group of rank $4$ is a CI-group,
J. Combin. Theory Ser. A 94 (2001), 339--362.
%
\bibitem{J}
G.\,A. Jones, 
Cyclic regular subgroups of primitive permutation groups, 
J. Group Theory 5 (2002), 403--407.
%
\bibitem{KP}
M.\,H. Klin, R. P\"oschel,
The K\"onig problem, the isomorphism problem for cyclic graphs and the method of Schur rings, 
in: Algebraic methods in graph theory (Szeged, 1978), Colloq. Math. Soc.
J\'anos Bolyai 25, North-Holland, Amsterdam--New York, 1981, 405--434.
%
\bibitem{KMPRS}
I. Kov\'acs, M. Muzychuk, P.\,P. P\'alfy, G. Ryabov, G. Somlai, 
CI-property of $C_p^2 \times C_n$ and $C_p^2 \times C_q^2$ 
for digraphs, 
J. Combin. Theory Ser. A 94 (2023), 105738.
%
\bibitem{L02}
C.\,H. Li,
On isomorphisms of finite Cayley graphs -- a survey,
Discrete Math. 256 (2002), 301--334.
%
\bibitem{L}
C.\,H. Li,
Permutation groups with a cyclic regular subgroup and arc-transitive circulants, 
J. Algebraic Combin. 21 (2005), 131--136.
%
\bibitem{LLP}
C.\,H. Li, Z.\,P.~Lu, P.\,P. P\'alfy,
Further restrictions on the structure of finite CI-groups,
J. Algebraic Combin. 26 (2007), 161--181.
%
\bibitem{M23}
J. Morris,
Dihedral groups of order $2pq$ and $2pqr$ are DCI, 
preprint arXiv:2311.12277v1 (https: //arxiv.org/abs/2311.12277). 
%
\bibitem{M97}
M. Muzychuk,
On \'Ad\'am's conjecture for circulant graphs,
Discrete Math. 167/168 (1997), 497--510; 
corrigendum 176 (1997), 285--298.
%
\bibitem{M99}
M. Muzychuk,
On the isomorphism problem for cyclic objects, 
Discrete Math. 197/198 (1999), 589--606.
%
\bibitem{MP}
M. Muzychuk, I. Ponomarenko,
Schur rings, 
European J. Combin. 30 (2009), 1526--1539.
%
\bibitem{Sch}
I. Schur,
Zur Theorie der einfach transitiven Permutationgruppen,
S.-B. Preuss. Akad. Wiss. Phys.-Math. Kl. 18 (1933), 598--623.
%
\bibitem{SM}
G. Somlai, M. Muzychuk,
The Cayley isomorphism property for $\Z_p^3 \times \Z_q$,
Algebr. Comb. 4 (2021), 289--299.
%
\bibitem{SLZ}
S.\,J. Song, C.\,H. Li, H. Zhang,
Finite permutation groups with a regular dihedral subgroup, and edge-transitive dihedrants, 
J. Algebra 399 (2014), 948--959.
%
\bibitem{Wbook}
H. Wielandt,
Finite permutation groups,
Academic Press, New York -- London, 1964.
%
\bibitem{Zs}
K. Zsigmondy, 
Zur Theorie der Potenzreste, 
Monatsch. Math. Phys. 3 (1892), 265--284.
%
\bibitem{XFK}
J-H. Xie, Y-Q. Feng, Y.\,S. Kwon,
Dihedral groups with the $m$-DCI-property,
J. Algebraic Combin. 60 (2024), 73--86. 
\end{thebibliography}
\end{document}